%% file: main.tex
\documentclass[12pt]{article}
\usepackage[margin=1in]{geometry}

\usepackage{amsmath,amsthm,amssymb,amsfonts}
\usepackage{algorithm,algorithmic} %2e
\usepackage{mathtools}
\usepackage{tcolorbox}
\usepackage{bbm}
\usepackage{todonotes}
\usepackage{tensor}
\usepackage{caption}

\newcommand{\RR}{\mathbb{R}}

\usepackage{accents}

\mathchardef\mhyphen="2D % Define a "math hyphen"
\DeclareMathOperator*{\argmin}{\arg\!\min}

\newtheorem{theorem}{Theorem}
\numberwithin{theorem}{section}

\newtheorem{lemma}[theorem]{Lemma}

\newtheorem{corollary}[theorem]{Corollary}

\let\originalleft\left
\let\originalright\right
\renewcommand{\left}{\mathopen{}\mathclose\bgroup\originalleft}
\renewcommand{\right}{\aftergroup\egroup\originalright}

\usepackage{titlesec} % or RequirePackage[loadonly]{titlesec} in a cls file
\titleformat{\subsubsection}[runin]{\normalfont\bfseries}{\thesubsubsection.}{3pt}{}
\begin{document}
	\title{General Convergence Rates Follow From Specialized Rates Assuming Growth Bounds}
	\author{Benjamin Grimmer\footnote{School of Operations Research and Information Engineering, Cornell University,
			Ithaca, NY 14850, USA;
			\texttt{people.orie.cornell.edu/bdg79/} \newline This material is based upon work supported by the National Science Foundation Graduate Research Fellowship under Grant No. DGE-1650441.}}
	\date{}
	\maketitle

	\begin{abstract}
		Often in the analysis of first-order methods, assuming the existence of a quadratic growth bound (a generalization of strong convexity) facilitates much stronger convergence analysis. Hence the analysis is done twice, once for the general case and once for the growth bounded case. We give a meta-theorem for deriving general convergence rates from those assuming a growth lower bound. Applying this simple but conceptually powerful tool to the proximal point method, the subgradient method, and the bundle method immediately recovers their known convergence rates for general convex optimization problems from their specialized rates. Future works studying first-order methods can assume growth bounds for the sake of analysis without hampering the generality of the results.
		Our results can be applied to lift any rate based on a H\"older growth bound. As a consequence, guarantees for minimizing sharp functions imply guarantees for both general functions and those satisfying quadratic growth.
	\end{abstract}
%	\newpage
%	\setcounter{tocdepth}{2}
%	\tableofcontents
%	\newpage
	\input{introduction}

	\input{proximal}

	\input{subgradient}

	\input{bundle}

	\paragraph{Acknowledgments.} The author would like to thank Adrian Lewis for advising on the positioning of this work and Calvin Wylie for directing his attention to~\cite{Kiwiel2000}, which has been overlooked in some of the recent bundle method literature. 
	
	\bibliographystyle{plain}
	{\small \bibliography{references}}
\end{document}

%% file: introduction.tex
\section{Introduction}
Throughout the literature on first-order optimization methods, improved convergence rates typically follow from assuming the given objective function possesses a growth lower bound. Consider the nonsmooth optimization problem
\begin{equation}
	\min_{x\in\RR^n} F(x)
\end{equation}
for some convex $F\colon \RR^n\rightarrow \RR$ that attains its minimum value at some $x^*$.  Often convergence guarantees can be improved by assuming a quadratic growth lower bound~\cite{Necoara2019}
\begin{equation}\label{eq:quadratic-growth}
	F(x) \geq F(x^*) + \alpha\|x-x^*\|^2
\end{equation}
for some $\alpha>0$, which is a common generalization of strong convexity. Even stronger convergence guarantees are often made possible by instead assuming a sharp growth lower bound
\begin{equation}\label{eq:sharp-growth}
F(x) \geq F(x^*) + \alpha\|x-x^*\|.
\end{equation}

Typically different convergence proofs are required for the cases of minimizing $F$ with and without assuming the existence of each growth lower bound. For example, the following table summarizes the number of iterations required to reach $\epsilon$-accuracy for three standard methods for nonsmooth optimization: the {\it Proximal Point Method} with stepsizes $\rho_k>0$ defined by
\begin{equation}\label{eq:proximal}
	x_{k+1} = \mathrm{prox}_{\rho_k, F}(x_{k})
\end{equation}
where $\mathrm{prox}_{\rho, F}(x_k) = \argmin\{F(\cdot) + \frac{1}{2\rho}\|\cdot-x_k\|^2\}$ denotes the proximal operator, the {\it Subgradient Method} with stepsizes $\rho_k>0$ defined by
\begin{equation}\label{eq:subgradient}
	x_{k+1} = x_{k}-\rho_kg_k\text{ for some } g_k\in\partial F(x_k)
\end{equation}
where $\partial F(x) = \{g\in\RR^n \mid F(y) \geq F(x) + \langle g, y-x\rangle\ \forall y\in\RR^n\}$ denotes the subdifferential of $F$,
and the {\it Bundle Method} which is defined later in Algorithms~\ref{alg:multiple-cuts} and~\ref{alg:aggregation}.
	\begin{center}
	\renewcommand{\arraystretch}{1.5}
	\begin{tabular}{r | c | c | c}
		Method & General Rate & Quadratic Growth Rate & Sharp Growth Rate\\
		\hline
		Proximal Point& $O(\|x_0-x^*\|^2/\epsilon)$ & $O(\log(1/\epsilon)/\alpha)$ & $O((F(x_0)-F(x^*))/\alpha^2)$\\
		Subgradient & $O(\|x_0-x^*\|^2/\epsilon^{2})$& $O(1/\epsilon\alpha)$ & $O(\log(1/\epsilon)/\alpha^2)$\\
		Bundle & $O(\|x_0-x^*\|^4/\epsilon^{3})$ & $O(\log(1/\epsilon)/\epsilon\alpha^2)$ & N/A
	\end{tabular}
\end{center}
Each of these claimed convergence rates is formalized in Sections~\ref{sec:proximal},~\ref{sec:subgradient}, and~\ref{sec:bundle} for the proximal point method, subgradient method, and bundle method, respectively.

The primary contribution of this paper is a pair of meta-theorems for deriving general convergence rates from rates that assume the existence of a growth lower bound. In terms of the above table, we show that each convergence rate implies all of the convergence rates to its left (up to logarithmic terms). This means that the quadratic growth column's rates imply the all of the general setting's rates. Likewise, each convergence rate under sharp growth implies that method's general and quadratic growth rate. Thus deriving a convergence guarantee for sharp functions bounds the method's behavior in all three cases.

Our proposed technique for lifting convergence rates is independent of the details of the first-order method being considered and critically does not require any modification to the method. As a result, we believe this is a strong addition to the optimization tool belt for studying first-order methods. One can assume growth bounds throughout their analysis without losing generality in terms of the final convergence rates. This shortcut is especially useful for the analysis of more sophisticated methods (like the bundle method) where convergence proofs become increasingly complex.

\subsection{Lifting Specialized Convergence Rates}
We prove our convergence rate lifting theorems using a generalization of quadratic and sharp growth. We say a function $F$ satisfies a H\"older growth bound for some $\alpha>0$ and $p\geq 1$ if
\begin{equation}\label{eq:holder-growth}
F(x) \geq F(x^*) + \alpha\|x-x^*\|^p.
\end{equation}
This is exactly quadratic growth when $p=2$ and sharpness when $p=1$. A number of recent works have focused on showing improved convergence rates for various first-order methods in the presence of H\"older growth~\cite{Johnstone2019,RenegarGrimmer2018,Roulet2017}.

Now we formalize our model for some first-order method $\mathtt{fom}$ that produces a sequence of iterates $\{x_k\}_{k=0}^\infty$ (in our examples, $\mathtt{fom}$ is one of the proximal point, subgradient, or bundle methods). Note that for it to be meaningful to lift a convergence rate to apply to general problems, the inputs to $\mathtt{fom}$ need to be independent of the existence of a growth bound~\eqref{eq:holder-growth}. 
We make the following three assumptions about $\mathtt{fom}$:
\begin{itemize}
	\item[(A1)] Each iteration $k$ of \texttt{fom} computes a point $x_{k+1}$ based on the values of $F$ and $\partial F$ at the set of points $\{x_j\}_{j=0}^{k+1}$.
	\item[(A2)] The distance from any $x_k$ to some fixed $x^*\in\argmin F$ is at most some constant $D>0$.
	\item[(A3)] For some fixed $p\geq 1$, if $F$ satisfies the growth condition~\eqref{eq:holder-growth} with $\alpha>0$, then for any $\epsilon>0$, {\normalfont \texttt{fom}} will an $\epsilon$-minimizer $x_k$ with $k \leq K(x_0,\epsilon,\alpha)$.
\end{itemize}
\paragraph{On the Generality of (A1)-(A3).}
Note that (A1) allows the computation of $x_{k+1}$ to depend on the function and subgradient values at $x_{k+1}$. Hence the proximal point method is included in our model since it is equivalent to $x_{k+1} = x_{k} - \rho_kg_{k+1},\text{ where }g_{k+1}\in\partial F(x_{k+1})$.

We remark that (A2) holds for all three of the proximal point, subgradient, and bundle methods under proper selection of their stepsize parameters. For each method, we provide an upper bound on the constant $D$ later in Lemmas~\ref{lem:proximal-distance}, \ref{lem:subgradient-distance}, and~\ref{lem:bundle-distance}, respectively.

Lastly, notice that (A3) allows the convergence bound $K(x_0,\epsilon,\alpha)$ to depend on $x_0$, which defines the entire sequence of iterates. Thus $K(x_0,\epsilon,\alpha)$ may depend on the initial distance $\|x_0-x^*\|$ and objective gap $F(x_0) - F(x^*)$ as well as constants like $\sup \|g_k\|$ for the subgradient and bundle methods which is at most the Lipschitz constant of $F$ around the iterates.\\

The following convergence rate lifting theorem shows that a general convergence guarantee follows from these three assumptions.
\begin{theorem}[Lifting To General Setting]\label{thm:rate-lifting}
	Consider any method $\mathtt{fom}$ satifying (A1)-(A3) and any function $F$ attaining its minimum somewhere. For any $\epsilon>0$, {\normalfont \texttt{fom}} will find some $x_k$ that is an $\epsilon$-minimizer of $F$ with $k\leq K(x_0,\epsilon,\epsilon/D^p)$.
\end{theorem}
\begin{proof}
	Suppose that no iterate $x_k$ with $k\leq T$ is an $\epsilon$-minimizer of $F$. 
	Consider the auxiliary function $G(x) = \max\{F(x), F(x^*) + \frac{\epsilon}{D^p}\|x - x^*\|^p\}$. Then (A2) ensures that each $k\leq T$ has 
	$$F(x^*) + \frac{\epsilon}{D^p}\|x_k - x^*\|^p \leq F(x^*) + \epsilon < F(x_k).$$
	Thus around all of the points $\{x_k\}_{k=0}^T$ considered by \texttt{fom}, $G(\cdot) = F(\cdot)$ and $\partial G(\cdot) = \partial F(\cdot)$.
	
	From this, (A1) ensures that applying \texttt{fom} to either $G$ or $F$ produces the same sequence of iterates up to iteration $T$. Since $F$ and $G$ both minimize at $x^*$, we know that no $x_k$ with $k\leq T$ is an $\epsilon$-minimizer of $G$ either. By definition, $G$ satisfies the H\"older growth condition~\eqref{eq:holder-growth} with exponent $p$ and coefficient $\epsilon/D^p$. Hence (A3) implies that $T < K(x_0,\epsilon,\epsilon/D^p)$.
\end{proof}

The key observation underlying this proof is that the given function $F$ cannot be discerned from the auxiliary function $G$ unless an $\epsilon$-accuracy iterate is found. Thus small amounts of H\"older growth can be assumed without effecting the sequence of iterates produced by the algorithm.
A general convergence rate is then given by simply substituting $\alpha$ with $\epsilon/D^p$.

Similarly, any rate based on H\"older growth with exponent $p$ implies rates for that method under any H\"older growth with exponent $q>p$. In terms of our previous table, this shows that convergence guarantees in the sharp column (where $p=1$) imply the guarantees in the quadratic growth column (where $q=2$).

\begin{theorem}[Lifting To Higher-Order Growth Settings]\label{thm:higher-order-rate-lifting}
	Consider any method $\mathtt{fom}$ satisfying (A1) and (A3) and any function $F$ satisfying the growth condition~\eqref{eq:holder-growth} with $q>p$ and $\alpha>0$. For any $\epsilon>0$, {\normalfont \texttt{fom}} will find some $x_k$ that is an $\epsilon$-minimizer of $F$ with $k\leq K(x_0,\epsilon,\alpha^{p/q}\epsilon^{1-p/q})$.
\end{theorem}
\begin{proof}
	Suppose that no iterate $x_k$ with $k\leq T$ is an $\epsilon$-minimizer of $F$. 
	Consider the auxiliary function $G(x) = \max\{F(x), F(x^*) + \alpha^{p/q}\epsilon^{1-p/q}\|x - x^*\|^p\}$. We begin by showing that $F(x^*) + \alpha^{p/q}\epsilon^{1-p/q}\|x_k - x^*\|^p < F(x_k)$ for all $k\leq T$.  If $\|x_k-x^*\| > (\epsilon/\alpha)^{1/q}$, then
	$$ F(x^*) + \alpha^{p/q}\epsilon^{1-p/q}\|x - x^*\|^p < F(x^*) + \alpha \|x_k-x^*\|^q \leq F(x_k).$$
	If $\|x_k-x^*\|\leq (\epsilon/\alpha)^{1/q}$, then
	$$ F(x^*) + \alpha^{p/q}\epsilon^{1-p/q}\|x - x^*\|^p \leq F(x^*) + \epsilon < F(x_k).$$
	Thus around all of the points $\{x_k\}_{k=0}^T$ considered by \texttt{fom}, $G(\cdot) = F(\cdot)$ and $\partial G(\cdot) = \partial F(\cdot)$.
	
	From this, (A1) ensures that applying \texttt{fom} to either $G$ or $F$ produces the same sequence of iterates up to iteration $T$. Since $F$ and $G$ both minimize at $x^*$, we know that no $x_k$ with $k\leq T$ is an $\epsilon$-minimizer of $G$ either. By definition, $G$ satisfies the H\"older growth condition~\eqref{eq:holder-growth} with exponent $p$ and coefficient $\alpha^{p/q}\epsilon^{1-p/q}$. Hence (A3) implies that $T < K(x_0,\epsilon,\alpha^{p/q}\epsilon^{1-p/q})$.
\end{proof}

This theorem relies on the same type of reasoning underlying our general lifting result. The given function $F$ cannot be discerned from the auxiliary function $G$ unless an $\epsilon$-accuracy iterate is found.
Applying this to lift sharp convergence rates to apply to functions satisfying quadratic growth simply requires substituting $\alpha$ with $\sqrt{\epsilon\alpha}$.

\paragraph{Extensions of our Lifting Theorems.} We make two remarks on straightforward extensions of our lifting theorems. First convergence rates can be derived by assuming growth away from the set of minimizers $X^*$ rather than assuming growth away from a unique minimizer $x^*$. Both of our lifting theorems still hold when~\eqref{eq:holder-growth} is replaced by
$$ F(x) \geq F(x^*) + \alpha\mathrm{dist}(x, X^*)^p.$$
Second notice that the convexity of $F$ is not used anywhere in the proof of either theorem. Hence if one can prove a convergence guarantee like (A3) for a nonconvex problem, then our theorems can still be applied to give more general convergence guarantees. 

\paragraph{Specialized Rates Following From General Rates.} Our two lifting theorems show that each entry in our previous convergence rate table implies the entries to its left. This prompts the question of whether the reverse implications to the right hold. If one is willing to modify the given first-order method, then the literature already provides a positive answer to this question through restarting schemes. Such schemes repeatedly run a given first-order method until some criteria is met and then restart the method at the current iterate. See~\cite{Nest1986} for the earliest proposal of a restarting method which relies on extensive knowledge of problem constants to determine when to restart and~\cite{RenegarGrimmer2018} for a recent restarting method which avoids assuming any knowledge about the problems structure.\\

In the remainder of this paper, we illustrate the usefulness of our lifting theorems as an analytic tool by applying them to the proximal point, subgradient, and bundle methods. We focus on the cases of $p=1$ and $p=2$ where H\"older growth becomes sharpness and quadratic growth, respectively. For the proximal point and subgradient methods, we include convergence proofs under sharpness and quadratic growth since they are particularly short.

%% file: proximal.tex
\section{Application to the Proximal Point Method}\label{sec:proximal}
First we use our lifting theorem to derive a general convergence rate for the proximal point method, defined by~\eqref{eq:proximal}. The standard convergence analysis of this method with constant stepsize $\rho_k=\rho>0$ shows that for any convex $F$ with minimizer $x^*$, an $\epsilon$-minimizer $x_k$ is found with
$$ k\leq O\left(\frac{\|x_0-x^*\|^2}{\rho\epsilon}\right).$$

To apply our machinery, we only need to verify that (A1)-(A3) hold. We begin by showing (A1) and (A2). Then we give two convergence proofs assuming H\"older growth holds with $p=2$ and with $p=1$. Using either of these to establish (A3) allows us to recover the standard convergence rate of the proximal point method in general (with an addition logarithmic term when $p=2$). 

\paragraph{Verifying (A1).} As previously noted, the proximal point method is equivalent to
$$x_{k+1} = x_{k} - \rho_kg_{k+1},\text{ where }g_{k+1}\in\partial F(x_{k+1}),$$
which is exactly the first-order optimality condition of the subproblem defining each iteration. Since we allow the computation of $x_{k+1}$ to depend on the function and subgradient values at $x_{k+1}$, the proximal point method satisfies (A1).

\paragraph{Verifying (A2).} Standard properties of the proximal operator give the following bound.
\begin{lemma}\label{lem:proximal-distance}
	For any minimizer $x^*$ and $k\geq0$, $\|x_k-x^*\| \leq \|x_0-x^*\|$. Hence (A2) holds with $D=\|x_0-x^*\|$.
\end{lemma}
\begin{proof}
	The proximal operator $\mathrm{prox}_{\rho, F}(\cdot)$ is nonexpansive~\cite{Parikh2014}. Then since any minimizer $x^*$ of $F$ is a fixed point of $\mathrm{prox}_{\rho, F}(\cdot)$, the distance from each iterate to $x^*$ must be nonincreasing.
\end{proof}

\paragraph{Verifying (A3) with $p=2$.} Assuming a quadratic growth lower bound allows us to guarantee the proximal point method converges linearly. Below we compute the function $K(x_0,\epsilon,\alpha)$ corresponding to this rate.
\begin{theorem}\label{thm:proximal-quadratic}
	Consider any convex $F$ satisfying~\eqref{eq:holder-growth} with $p=2$. Then for any $\epsilon>0$, the proximal point method with constant stepsize $\rho_k=\rho>0$ will find an $\epsilon$-minimizer $x_k$ with 
	$$ k\leq K(x_0,\epsilon,\alpha) = \frac{1}{\log\left(1 + \frac{\alpha\rho}{2}\right)}\log\left(\frac{F(x_0)-F(x^*)}{\epsilon}\right).$$
\end{theorem}
\begin{proof}
	The definition of the proximal subproblem ensures each $k\geq 0$ has
	$$ F(x_{k+1}) - F(x^*) \leq F(x_{k}) - F(x^*) - \frac{\|x_{k+1}-x_{k}\|^2}{2\rho}.$$
	The optimality condition of the proximal subproblem is $(x_{k+1}-x_{k})/\rho \in \partial F(x_{k+1})$. Hence
	$$ \|x_{k+1}-x_{k}|\|x_{k+1}-x^*\|/\rho \geq \langle (x_{k+1}-x_{k})/\rho, x_{k+1} - x^*\rangle \geq F(x_{k+1}) - F(x^*).$$
	Then supposing that $x_k\neq x^*$, it follows from the assumed quadratic growth bound that
	$$\|x_{k+1}-x_{k}\|^2 \geq \rho^2 (F(x_{k+1}) - F(x^*))^2/\|x_{k+1}-x^*\|^2 \geq \rho^2\alpha(F(x_{k+1}) - F(x^*)).$$
	Thus we have a geometric decrease in the objective gap at each iteration of
	\begin{equation*}
		F(x_{k+1}) - F(x^*) \leq \left(1 + \frac{\rho\alpha}{2}\right)^{-1}(F(x_{k}) - F(x^*)). \qedhere
	\end{equation*}
	%and so there are at most $\frac{2}{\rho\alpha}\log\left(\frac{F(x_0)-F(x^*)}{\epsilon}\right)$ iterations before an $\epsilon$-minimizer is found.
\end{proof}

Applying Theorem~\ref{thm:rate-lifting} with this rate of $K(x_0,\epsilon,\alpha) = O(\log((F(x_0)-F(x^*))/\epsilon)/\rho\alpha)$ gives a general rate of $\widetilde O(\|x_0-x^*\|^2/\rho\epsilon)$, matching the method's known convergence rate up to a logarithmic term\footnote{We use $\widetilde O(\cdot)$ to hide logarithmic terms.}.
\begin{corollary}\label{cor:proximal-rate2}
	Consider any convex $F$ that attains its minimum value. Then for any $\epsilon>0$, the proximal point method with constant stepsize $\rho_k=\rho>0$ will find an $\epsilon$-minimizer $x_k$ with 
	$$k \leq \frac{1}{\log\left(1 + \dfrac{\rho\epsilon}{2\|x_0-x^*\|^2}\right)}\log\left(\frac{F(x_0)-F(x^*)}{\epsilon}\right).$$
\end{corollary}

% From this, applying Theorem~\ref{thm:rate-lifting} recovers the known general convergence rate (without any additional logarithmic terms).

\paragraph{Verifying (A3) with $p=1$.} Rather than relying on quadratic growth, assuming sharpness (H\"older growth with $p=1$) allows us to derive a finite termination bound on the number of iterations before an exact minimizer is found (see~\cite{Ferris1991} for a more general proof and discussion of this finite result). Below we compute the resulting function $K(x_0,\alpha,\epsilon)$ that satisfies (A3).
\begin{theorem}\label{thm:proximal-sharp}
	Consider any convex $F$ satisfying~\eqref{eq:holder-growth} with $p=1$. Then for any $\epsilon>0$, the proximal point method with constant stepsize $\rho_k=\rho>0$ will find an $\epsilon$-minimizer $x_k$ with 
	$$ k\leq K(x_0,\epsilon,\alpha) = 2\frac{F(x_0)-F(x^*)}{\rho\alpha^2}.$$
\end{theorem}
\begin{proof}
	We prove this by showing an exact minimizer of $F$ is found by iteration $K(x_0,\epsilon,\alpha)$. The definition of the proximal subproblem solved at each iteration $k\geq 0$ ensures
	$$ F(x_{k+1}) - F(x^*) \leq F(x_{k}) - F(x^*) - \frac{\|x_{k+1}-x_{k}\|^2}{2\rho}.$$
	The optimality condition of the proximal subproblem is $(x_{k+1}-x_{k})/\rho \in \partial F(x_{k+1})$. Hence
	$$ \|x_{k+1}-x_{k}|\|x_{k+1}-x^*\|/\rho \geq \langle (x_{k+1}-x_{k})/\rho, x_{k+1} - x^*\rangle \geq F(x_{k+1}) - F(x^*).$$
	Then supposing that $x_k\neq x^*$, the assumed sharp growth bound implies
	$$\|x_{k+1}-x_{k}\| \geq \rho (F(x_{k+1}) - F(x^*))/\|x_{k+1}-x^*\| \geq \rho\alpha.$$
	Thus we have a constant decrease in objective each iteration of
	\begin{equation*}
		F(x_{k+1}) - F(x^*) \leq F(x_{k}) - F(x^*) - \frac{\rho\alpha^2}{2}. \qedhere
	\end{equation*}
%	and so there are at most $2(F(x_0)-F(x^*))/\rho\alpha^2$ iterations before a minimizer is found.
\end{proof}

Applying each of our Lifting Theorems~\ref{thm:rate-lifting} and~\ref{thm:higher-order-rate-lifting} to this recovers the general convergence rate and quadratic growth rate of the proximal point method.
\begin{corollary}\label{cor:proximal-rate1}
	Consider any convex $F$ that attains its minimum value. Then for any $\epsilon>0$, the proximal point method with constant stepsize $\rho_k=\rho>0$ will find an $\epsilon$-minimizer $x_k$ with 
	$$k \leq \frac{2(F(x_0)-F(x^*))\|x_0-x^*\|^2}{\rho\epsilon^2}.$$
	Setting $\epsilon=(F(x_0)-F(x^*))/2$ above shows the objective gap halves by iteration
	$$\frac{8\|x_0-x^*\|^2}{\rho(F(x_0)-F(x^*))}.$$
	Thus for any $\epsilon>0$, an $\epsilon$-minimizer $x_k$ is found with
	\begin{align*}
	k&\leq \frac{8\|x_0-x^*\|^2}{\rho}\left(\frac{1}{F(x_0)-F(x^*)} + \frac{2}{F(x_0)-F(x^*)} + \frac{4}{F(x_0)-F(x^*)} + \dots + \frac{1}{\epsilon}\right)\\
	&\leq \frac{16\|x_0-x^*\|^2}{\rho\epsilon}.
	\end{align*}
\end{corollary}
\begin{corollary}\label{cor:proximal-rate12}
	Consider any convex $F$ satisfying~\eqref{eq:holder-growth} with $p=2$. Then for any $\epsilon>0$, the proximal point method with constant stepsize $\rho_k=\rho>0$ will find an $\epsilon$-minimizer $x_k$ with 
	$$k \leq \frac{2(F(x_0)-F(x^*))}{\rho\alpha\epsilon}.$$
	Setting $\epsilon=(F(x_0)-F(x^*))/2$ above shows the objective gap halves by iteration
	$$\frac{4}{\rho\alpha}.$$
	Thus for any $\epsilon>0$, an $\epsilon$-minimizer $x_k$ is found with
	\begin{align*}
	k&\leq \frac{4}{\rho\alpha}\log\left(\frac{F(x_0)-F(x^*)}{\epsilon}\right).
	\end{align*}
\end{corollary}

%% file: subgradient.tex
\section{Application to the Subgradient Method}\label{sec:subgradient}
Now we use our lifting theorem to derive a general convergence rate for the subgradient method given in~\eqref{eq:subgradient}. We consider using the Polyak stepsize which is defined by
$$ \rho_k = \frac{F(x_k)-F(x^*)}{\|g_k\|^2}.$$ 
We use this stepsize since it does not rely on any knowledge of a H\"older growth bound~\eqref{eq:holder-growth}. This does however mean that the method relies on the minimum objective value $F(x^*)$ being known in advance.

The standard analysis of this method requires a uniform bound on the norm of all of the subgradients used. To this end, we will assume the following constant exists
$$ L =\sup \|g_k\|,$$
which must exist if the objective function is uniformly Lipschitz continuous.
Then for any convex $F$ with minimizer $x^*$, the subgradient method will find an $\epsilon$-minimizer $x_k$ with
$$ k\leq O\left(\frac{L^2\|x_0-x^*\|^2}{\epsilon^2}\right).$$

Applying our machinery only requires that we verify (A1)-(A3) hold. First we establish (A1) and (A2) and then give two convergence proofs assuming H\"older growth holds with $p=2$ and with $p=1$, establishing (A3).

\paragraph{Verifying (A1).} Each iteration computes a single subgradient at the current iterate and uses the current objective value to compute the stepsize. Hence (A1) holds.

\paragraph{Verifying (A2).} Much like the proximal point method, the distance from each iterate of the subgradient method to a minimizer is nonincreasing when the Polyak stepsize is used.
\begin{lemma}\label{lem:subgradient-distance}
	For any minimizer $x^*$ and $k\geq0$, $\|x_k-x^*\| \leq \|x_0-x^*\|$. Hence (A2) holds with $D=\|x_0-x^*\|$.
\end{lemma}
\begin{proof}
	The convergence of the subgradient method is governed by the following inequality
	\begin{align}
		\|x_{k+1} - x^*\|^2 &= \|x_{k}-x^*\|^2 -2\langle\rho_kg_k, x_{k}-x^*\rangle + \rho_k^2\|g_k\|^2\nonumber\\
		& \leq \|x_{k}-x^*\|^2 -2\rho_k (F(x_k)-F(x^*)) + \rho_k^2\|g_k\|^2\nonumber\\
		& = \|x_{k}-x^*\|^2 -\frac{(F(x_k)-F(x^*))^2}{\|g_k\|^2}\nonumber\\
		& \leq \|x_{k}-x^*\|^2 -\frac{(F(x_k)-F(x^*))^2}{L^2}, \label{eq:subgradient-recurrence}
	\end{align}
	where the first inequality uses the convexity of $F$ and the second uses our assumed subgradient bound.
	Thus the distance from each iterate to $x^*$ is nonincreasing.
\end{proof}

\paragraph{Verifying (A3) with $p=2$.} Below we give a simple proof that the subgradient method under the quadratic growth condition finds an $\epsilon$-minimizer within $\widetilde O(L^2/\alpha\epsilon)$ iterations. We note that a more refined analysis (using an approach like~\cite{Lacoste2012}) would likely remove the logarithmic term.
\begin{theorem}\label{thm:subgradient-quadratic}
	Consider any convex $F$ satisfying~\eqref{eq:holder-growth} with $p=2$. Then for any $\epsilon>0$, the subgradient method with the Polyak stepsize will find an $\epsilon$-minimizer $x_k$ with 
	$$ k\leq K(x_0,\epsilon,\alpha) = \frac{2L^2}{\alpha\epsilon}\log\left(\frac{L\|x_0-x^*\|}{\epsilon}\right).$$
\end{theorem}
\begin{proof}
	Suppose after $T$ iterations, no $x_k$ with $k\leq T$ is an $\epsilon$-minimizer. Then~\eqref{eq:subgradient-recurrence} implies
	\begin{align*}
	\|x_{k+1} - x^*\|^2 &\leq \|x_{k} - x^*\|^2 - \frac{\epsilon(F(x_k)-F(x^*))}{L^2}\\
	&\leq \|x_{k} - x^*\|^2 - \frac{\alpha\epsilon}{L^2}\|x_{k}-x^*\|^2
	\end{align*}
	where the first inequality uses that $x_k$ is not an $\epsilon$-minimizer and the second uses the assumed quadratic growth bound.
	Thus $\|x_{k}-x^*\|$ is decreasing geometrically. However, the assumed subgradient norm bound ensures
	$$L\|x_T-x^*\| \geq \langle g_k, x^*-x_k\rangle \geq F(x_k)-F(x^*)\geq \epsilon.$$
	Thus
	\begin{equation*}
		T \leq \frac{2L^2}{\alpha\epsilon}\log\left(\frac{L\|x_0-x^*\|}{\epsilon}\right). \qedhere
	\end{equation*}
%	Let $R_k = \|x_k-x^*\|^2$ denote the distance squared from each iterate to $x^*$.
%	Together~\eqref{eq:subgradient-recurrence} and the assumed quadratic growth lower bound yield the following recurrence relation
%	$$ R_{k+1} \leq R_{k} -\frac{\alpha^2R_{k}^2}{L^2}.$$
%	Solving this recurrence implies $R_{k}\leq L^2/\alpha^2(k+1)$. Then applying the $L$-Lipschitz continuity of $F$ to the upper bound of $\|x_k - x^*\| \leq L/\alpha(k+1)$ implies
%	\begin{equation*}
%		F(x_k) - F(x^*) \leq \frac{L^2}{\alpha(k+1)}. \qedhere
%	\end{equation*}
\end{proof}

Applying Theorem~\ref{thm:rate-lifting} recovers the method's general rate (up to a logarithmic term).
\begin{corollary}\label{cor:subgradient-rate2}
	Consider any convex $F$ that attains its minimum value. Then for any $\epsilon>0$, the subgradient method with the Polyak stepsize will find an $\epsilon$-minimizer $x_k$ with 
	$$k \leq \frac{2L^2\|x_0-x^*\|^2}{\epsilon^2}\log\left(\frac{L\|x_0-x^*\|}{\epsilon}\right).$$
\end{corollary}

\paragraph{Verifying (A3) with $p=1$.} Rather than relying on quadratic growth, assuming sharpness (H\"older growth with $p=1$) allows us to derive a linear convergence guarantee. Below we compute the resulting function $K(x_0,\alpha,\epsilon)$ that satisfies (A3).
\begin{theorem}\label{thm:subgradient-sharp}
	Consider any convex $F$ satisfying~\eqref{eq:holder-growth} with $p=1$. Then for any $\epsilon>0$, the subgradient method with the Polyak stepsize will find an $\epsilon$-minimizer $x_k$ with 
	$$ k\leq K(x_0,\epsilon,\alpha) = \frac{2L^2}{\alpha^2}\log\left(\frac{L\|x_0-x^*\|}{\epsilon}\right).$$
\end{theorem}
\begin{proof}
	Suppose after $T$ iterations, no $x_k$ with $k\leq T$ is an $\epsilon$-minimizer. Together~\eqref{eq:subgradient-recurrence} and the sharp growth bound yield the following geometric decrease
	$$ \|x_{k+1}-x^*\|^2 \leq \|x_k-x^*\|^2 -\frac{\alpha^2}{L^2}\|x_k-x^*\|^2.$$
	Recall the assumed subgradient norm bound ensures $L\|x_T-x^*\| \geq \epsilon$. Thus
	\begin{equation*}
	T \leq \frac{2L^2}{\alpha^2}\log\left(\frac{L\|x_0-x^*\|}{\epsilon}\right). \qedhere
	\end{equation*}
\end{proof}

Applying our Lifting Theorems~\ref{thm:rate-lifting} and~\ref{thm:higher-order-rate-lifting} with this sharp convergence rate recovers the method's general rate as well as its rate under quadratic growth (up to logarithmic terms).
\begin{corollary}\label{cor:subgradient-rate1}
	Consider any convex $F$ that attains its minimum value. Then for any $\epsilon>0$, the subgradient method with the Polyak stepsize will find an $\epsilon$-minimizer $x_k$ with 
	$$k \leq \frac{2L^2\|x_0-x^*\|^2}{\epsilon^2}\log\left(\frac{L\|x_0-x^*\|}{\epsilon}\right).$$
\end{corollary}
\begin{corollary}\label{cor:subgradient-rate12}
	Consider any convex $F$ satisfying~\eqref{eq:holder-growth} with $p=2$. Then for any $\epsilon>0$, the subgradient method with the Polyak stepsize will find an $\epsilon$-minimizer $x_k$ with 
	$$k \leq \frac{2L^2}{\alpha\epsilon}\log\left(\frac{L\|x_0-x^*\|}{\epsilon}\right).$$
\end{corollary}

%% file: bundle.tex
\section{Application to the Bundle Method}\label{sec:bundle}
Lastly we consider applying our lifting theorems to bundle methods. These methods produce a sequence $\{x_k\}_{k=0}^\infty$ with non-increasing objective value by approximating $F$ with a cutting plane model $\widetilde F^k$ constructed from subgradients taken at a set of points $\{z_j\}_{j=0}^k$. %This approximation is used to produce a point $z^{k+1}$ which replaces $x^k$ as $x^{k+1}$ if it provides a sufficiently large decrease in objective value.
We consider two variations of the bundle method (stated in Algorithms~\ref{alg:multiple-cuts} and~\ref{alg:aggregation}) differing in how the approximation $\widetilde F^k$ is constructed.

Bundle methods were first developed and proposed independently in~\cite{Lemarechal1975} and~\cite{Wolfe1975}. Computationally simpler bundle methods which aggregate cuts in the model $\widetilde F^k$ were analyzed in~\cite{Kiwiel1983, Kiwiel1985}.
The central result in the convergence theory of bundle methods is that (for convex $F$ that attain their minimum value somewhere) the sequences $\{z_k\}$ and $\{x_k\}$ both converge to a minimizer $x^*$ of $F$. That is, the bundle method when run with no stopping criteria has
\begin{equation}\label{eq:limiting-convergence}
\lim_{k\rightarrow\infty} x_k = \lim_{k\rightarrow\infty} z_k = x^*\in\argmin F.
\end{equation}
See~\cite[Thm. 4.9]{Kiwiel1985}, \cite[Thm. XV.3.2.4]{Hiriart1993}, or~\cite[Thm. 7.16]{Ruszczynski2006} for proofs of different variations of this result.

A convergence rate theory for bundle methods exists as well, but is much weaker and less studied than those of the proximal point method and the subgradient method. In 2000, Kiwiel~\cite{Kiwiel2000} gave the first convergence rate for this method, showing that an $\epsilon$-minimizer $x_k$ is found with
$$k\leq O\left(\frac{\|x_0-x^*\|^4}{\epsilon^3}\right).$$
More recently, Du and Ruszczy\'nski~\cite{Du2017} gave the first analysis of bundle methods when applied to problems satisfying the quadratic growth bound. In this case, an $\epsilon$-minimizer is found within $\widetilde O(1/\alpha^2\epsilon)$ iterations. Hence this problem provides another interesting setting to apply our lifting theorem. In doing so, we find that the recent rate of Du and Ruszczy\'nski implies the rate of Kiwiel (up to logarithmic terms).

Despite having weaker convergence rate guarantees than either the proximal point or subgradient methods, bundle methods have persisted as a method of choice for nonsmooth convex optimization. In practice, bundle methods have proven to be efficient methods for solving many nonsmooth problems (see~\cite{Lemarechal2001,Sagastizabal2012,Sagastizabal2005} for further discussion). Extensions to apply to nonconvex problems have been considered in~\cite{Apkarian2009, Warren2010, Kiwiel1985-nonconvex, Mifflin1982} and an extension to problems where only an inexact first-order oracle was recently given by~\cite{Hare2016}.

Stronger convergence rates have been established for related level bundle methods~\cite{Lemarechal1995}, which share many core elements with bundle methods. Further variations of level bundle methods were studied in~\cite{Kiwiel1995} and~\cite{Lan2015}. The results of Lan~\cite{Lan2015} are particularly impressive as their proposed method has optimal convergence rates for both smooth and nonsmooth problems while requiring very little input.

\subsection{Defining the Bundle Method}
We utilize the same notation as Du and Ruszczy\'nski~\cite{Du2017} to describe the bundle method. The bundle method solves a sequence of proximal subproblems based on piecewise linear lower bounds of $F$ given by a collection (or bundle) of its subgradients. Let $\widetilde F^k(x)$ denote such an approximation of the objective function given by either a collection of cuts at the previous iterates or convex combinations of those cuts. Then iteration $k$ of the bundle method computes the solution $z_{k+1}$ of the proximal subproblem
\begin{equation}\label{eq:subproblem}
\min_{x\in\RR^n} \widetilde F^k(x) + \frac{\rho}{2}\|x - x_k\|^2
\end{equation}
where $x_k$ is the current best approximation to the solution and $\rho>0$ is a user chosen parameter. If the decrease in value of $F$ from $x_k$ to $z_{k+1}$ is at least $\beta$ fraction of the decrease expected of $F(x_k)-\widetilde F^k(z_{k+1})$, then the bundle method sets $x_{k+1} = z_{k+1}$ (called a descent step). Otherwise the method sets $x_{k+1} = x_k$ (called a null step). Regardless of which type of step is done, a subgradient $g_{k+1}\in\partial F(z_{k+1})$ is computed to update the approximation $\widetilde F^{k+1}$.

Two approaches to constructing the approximation $\widetilde F^k$ are formalized below. Both the convergence results of Du and Ruszczy\'nski as well as the results proven herein apply equally to each of these bundle methods.

\subsubsection{Bundle Method with Multiple Cuts}\label{subsec:multiple}
The bundle method with multiple cuts maintains a cutting plane approximation given by the subgradients taken at a subset of the iterations $J_k\subseteq \{1\dots k\}$, which is defined by
$$ \widetilde F^{k}(x) = \max_{j\in J_k}\left\{F(z_j) + \langle g_j, x - z_j\rangle\right\},$$
where each $g_j\in \partial F(z_j)$ is a subgradient of $F$ at $z_j$. Then this version of the method proceeds as defined by Algorithm~\ref{alg:multiple-cuts}, which establishes how the subset $J_k$ is chosen.
\begin{algorithm} 
	\caption{Bundle Method with Multiple Cuts} \label{alg:multiple-cuts}
	\begin{algorithmic}
		\STATE {\bf Input:} $F\colon\RR^n\rightarrow \RR$, $x_0\in\RR^n$, $\rho>0,\ \beta\in(0,1),\ \epsilon_{stop}\geq0$
		\STATE $J_0=\{0\}$, $z_0=x_0$ and select $g_0\in\partial F(z_0)$
		\FOR{$k=0, 1, 2, \dots$}
		\STATE Compute the solution $z_{k+1}$ of subproblem~\eqref{eq:subproblem}
		\STATE {\bf if} $F(x_k)-\widetilde F^{k}(z_{k+1}) \leq \epsilon_{stop}$ {\bf then} stop and output $x_k$ {\bf end if} \hfill {\it Stopping Criteria}
		\STATE {\bf if} $F(z_{k+1}) \leq F(x_k) - \beta\left(F(x_k) - \widetilde F^k(z_{k+1})\right)$ {\bf then} $x_{k+1} = z_{k+1}$ \hfill {\it Descent Step}
		\STATE {\bf else} $x_{k+1} = x_{k}$ {\bf end if} \hfill {\it Null Step}
		\STATE Select $g_{k+1} \in \partial F(z_{k+1})$ and a set $J_{k+1}$ so that $$ J_k\cup\{k+1\} \supseteq J_{k+1} \subseteq \{j\in J_k \mid F(z_j) + \langle g_j, z_{k+1}-z_j\rangle = \widetilde F^k(z_{k+1})\}$$
		\ENDFOR
	\end{algorithmic}
\end{algorithm}

\subsubsection{Bundle Method with Cut Aggregation}\label{subsec:aggregation}
The bundle method with cut aggregation only uses two cutting planes in its approximation of $F$. It uses the lower bound given by the most recent subgradient $F(z_k) + \langle g_k, x-z_k\rangle$ and a convex combination of the lower bounds given by the previous subgradients $\bar{F}^k(x)$. Then the approximation is defined by
$$ \widetilde F^{k}(x) = \max\{F(z_k) + \langle g_k, x-z_k\rangle, \bar{F}^k(x)\}.$$
Notice that subproblem~\eqref{eq:subproblem} now corresponds to minimizing the maximum of two quadratics, which can easily be done in closed form.
This version of the method proceeds as defined by Algorithm~\ref{alg:aggregation}, which establishes exactly how the convex combination $\bar{F}^k$ is chosen.
\begin{algorithm} 
	\caption{Bundle Method with Cut Aggregation} \label{alg:aggregation}
	\begin{algorithmic}
		\STATE {\bf Input:}  $F\colon\RR^n\rightarrow \RR$, $x_0\in\RR^n$, $\rho>0,\ \beta\in(0,1),\ \epsilon_{stop}\geq0$
		\STATE $\bar{F}^0(\cdot)=-\infty$, $z_0=x_0$ and select $g_0\in\partial F(z_0)$
		\FOR{$k=0, 1, 2,\dots$}
		\STATE Compute the solution $z_{k+1}$ of subproblem~\eqref{eq:subproblem}
		\STATE {\bf if} $F(x_{k})-\widetilde F^{k}(z_{k+1}) \leq \epsilon_{stop}$ {\bf then} stop and output $x_k$ {\bf end if} \hfill {\it Stopping Criteria}
		\STATE {\bf if} $F(z_{k+1}) \leq F(x_k) - \beta\left(F(x_k) - \widetilde F^k(z_{k+1})\right)$ {\bf then} $x_{k+1} = z_{k+1}$ \hfill {\it Descent Step}
		\STATE {\bf else} $x_{k+1} = x_{k}$ {\bf end if} \hfill {\it Null Step}
		\STATE Select $g_{k+1} \in \partial F(z_{k+1})$ and define $$ \bar{F}^{k+1}(x) = \theta_k \bar{F}^k(x) + (1-\theta_k)\left[F(z_k) + \langle g_k,x - z_k\rangle\right]$$
		where $\theta_k\in[0,1]$ is such that the gradient of $\bar{F}^{k+1}(x)$ is equal to the subgradient of $\widetilde F^k(\cdot)$ at $z_{k+1}$ that satisfies the optimality conditions for subproblem~\eqref{eq:subproblem}
		\ENDFOR
	\end{algorithmic}
\end{algorithm}

\subsection{Applying our Lifting Theorem}
To derive a general convergence rate for the bundle method from the rate of~\cite{Du2017} assuming quadratic growth, we only need to verify our three assumptions hold.

\paragraph{Verifying (A1)} The sequence of iterates where subgradients are taken by the bundle method is $\{z_k\}^\infty_{k=0}$ rather than the descent sequence $\{x_k\}^\infty_{k=0}\subseteq \{z_k\}^\infty_{k=0}$. Each iteration only relies the values of $F$ and $\partial F$ at the previous iterates $\{z_j\}_{j=0}^{k}$ to solve the subproblem~\eqref{eq:subproblem} and only uses the objective values at $x_k$ and $z_{k+1}$ to determine descent. Hence (A1) holds.

\paragraph{Verifying (A2)} Notice that the limiting convergence guarantee~\eqref{eq:limiting-convergence} ensures that the sequence of iterates has distance to their limit point $x^*$ is bounded by some constant $D>0$. Since this sequence is bounded, the following constant must exist
$$ L=\sup \|g_k\|.$$ 
From this, we compute the following distance bound, showing the distance to $x^*$ does not increase much above $\|x_0-x^*\|$.
\begin{lemma}\label{lem:bundle-distance}
	For any $k\geq 0$, $\|z_k-x^*\|^2 \leq O\left(\|x_0 - x^*\|^2 + L^2\right)$. Hence (A2) holds with $D^2=O\left(\|x_0 - x^*\|^2 + L^2\right)$.
\end{lemma}
\begin{proof}
	Notice that the optimality condition of the subproblem~\eqref{eq:subproblem} ensures that $-\rho(z_{k+1} - x_k)$ is a convex combination of the subgradients used to construct $\widetilde F^k$. Then the $L$-Lipschitz continuity of $\widetilde F^k$ (which follows from our uniform subgradient norm bound) implies each iteration $k$ has $ \rho\|z_{k+1} - x_k\| \leq L.$
	It follows from~\cite[(7.64)]{Ruszczynski2006} that
	$$ \|x_k-x^*\|^2 \leq \|x_0 - x^*\|^2 + \frac{2(1-\beta)}{\rho\beta}\left[F(x_0) - F(x^*)\right].$$
	Using the triangle inequality to combine these two inequalities yields
	\begin{align*}
	\|z_{k+1}-x^*\|^2 &\leq 2\|z_{k+1} - x_k\|^2 + 2\|x_k-x^*\|^2\\
	&\leq 2L^2/\rho^2 + 2\|x_0 - x^*\|^2 + \frac{4(1-\beta)}{\rho\beta}\left[F(x_0) - F(x^*)\right]\\
	&\leq 2L^2/\rho^2 + 2\|x_0 - x^*\|^2 + \frac{4(1-\beta)}{\rho\beta}L\|x_0-x^*\|\\
	&\leq 2\left(1+\frac{1-\beta}{\beta}\right)\left(\|x_0 - x^*\|^2 + \frac{L^2}{\rho^2}\right).
	\end{align*}
	where the last inequality uses that $2ab \leq a^2 +b^2$ for all $a,b\in\RR$.
	%Hence every iteration $k$ has $\|z^{k+1}-x^*\|^2 \leq O(\|x^1-x^*\|^2 + R^2)$.
\end{proof}

\paragraph{Verifying (A3) with $p=2$} Noting that the optimality condition of the subproblem~\eqref{eq:subproblem} ensures that $-\rho(z_{k+1} - x_k)$ is a convex combination of the subgradients used to construct $\widetilde F^k$, there must exist some constant $M>0$ such that each null step has
$$\|g_{k+1} - \rho(z_{k+1}-x_k)\|^2\leq\rho M.$$
Assuming a quadratic growth bound exists, Du and Ruszczy\'nski derive the following rate\footnote{We remark that there is a minor issue in the analysis of Du and Ruszczy\'nski. Lemma 3.4 of~\cite{Du2017} is missing the dependence on $\rho$ in the right-hand side of the inequality (see~\cite[Lem 7.12]{Ruszczynski2006} for the correct inequality). Including this dependence only minorly effects the constants in the derived rate of convergence: Lemma 4.1 of~\cite{Du2017} should then have $\bar{\alpha} = \min\{1, \alpha/\rho\}$ instead of $\min\{1, \alpha\}$. All of the subsequent results in the paper hold without modification (besides this change in the definition of the constant $\bar{\alpha}$).}.

\begin{theorem}[Du and Ruszczy\'nski~\cite{Du2017}]\label{thm:strong-full}
	Consider any convex $F$ satisfying~\eqref{eq:holder-growth} with $p=2$, and let $\bar{\alpha} = \min\{1, \alpha/\rho\}$ and $\eta_0 = \widetilde F^0(z_1) + \frac{\rho}{2}\|z_0-x_1\|^2$.
	Then the bundle method with stopping criteria $\epsilon_{stop}>0$ will terminate by iteration
	\begin{align*}
	\frac{2M}{(1-\beta)^2 \epsilon_{stop}}\ln\left(\frac{F(x_0) - \eta_0}{\epsilon_{stop}}\right) + \frac{\ln\left(\frac{F(x_0) - F(x^*)}{\beta\epsilon_{stop}}\right)}{\ln(1 - \bar{\alpha}\beta)}
	\left[\frac{2M}{(1-\beta)^2 \epsilon_{stop}}\ln\left(\frac{2\bar{\alpha}^2\beta^2}{9}\right) - 2\right] + 2
	\end{align*}
	and the last iterate $x_k$ will be an $\epsilon_{stop}/\bar{\alpha}$-minimizer.
\end{theorem}
We can simplify this bound slightly by using the following three inequalities. First, a simple calculation shows that $\eta_0 = F(x_0) - \|g_0\|^2/2\rho \geq F(x_0) - L^2/2\rho$. Second, since $\|g_{k+1} - \rho(z_{k+1}-x_k)\|^2\leq 4L^2$, the constant $M$ can be replaced by $4L^2/\rho$. Lastly, note that $\ln(1-\bar{\alpha}\beta) \leq -\bar{\alpha}\beta$.
Together these yield a bound of
\begin{align*}
\frac{8L^2}{\rho(1-\beta)^2\epsilon_{stop}}\ln\left(\frac{L^2}{2\rho\epsilon_{stop}}\right) + \ln\left(\frac{F(x_0) - F(x^*)}{\beta\epsilon_{stop}}\right)
\left[\frac{8L^2}{\rho(1-\beta)^2\beta \bar{\alpha}\epsilon_{stop}}\ln\left(\frac{9}{2\bar{\alpha}^2\beta^2}\right) + \frac{2}{\bar{\alpha}\beta}\right] + 2.
\end{align*}
A convergence rate for the bundle method with no stopping criteria can be extracted from this. For any $\epsilon>0$, an $\epsilon$-minimizer must be found by iteration
\begin{align*}
\frac{8L^2}{\rho(1-\beta)^2\bar{\alpha}\epsilon}\ln\left(\frac{L^2}{2\rho\bar{\alpha}\epsilon}\right) + \ln\left(\frac{F(x_0) - F(x^*)}{\beta\bar{\alpha}\epsilon}\right)
\left[\frac{8L^2}{\rho(1-\beta)^2 \beta\bar{\alpha}^2\epsilon}\ln\left(\frac{9}{2\bar{\alpha}^2\beta^2}\right) + \frac{2}{\beta\bar{\alpha}}\right] + 2
\end{align*}
as the bundle method with stopping criteria $\epsilon_{stop} = \bar{\alpha}\epsilon$ would have terminated.

Applying Theorem~\ref{thm:rate-lifting} with the above quantity as $K(x_0,\epsilon,\alpha)$ gives the following convergence guarantee for the bundle method when applied to a general convex function.
\begin{corollary}\label{cor:bundle-rate}
	Consider any convex $F$ that attains its minimum value. Then for any $\epsilon>0$, the bundle method with no stopping criteria will find an $\epsilon$-minimizer by iteration
	\begin{align*}
	\frac{8L^2}{\rho(1-\beta)^2\bar{\alpha}\epsilon}\ln\left(\frac{L^2}{2\rho\bar{\alpha}\epsilon}\right) + \ln\left(\frac{F(x_0) - F(x^*)}{\beta\bar{\alpha}\epsilon}\right)
	\left[\frac{8L^2}{\rho(1-\beta)^2 \beta\bar{\alpha}^2\epsilon}\ln\left(\frac{9}{2\bar{\alpha}^2\beta^2}\right) + \frac{2}{\beta\bar{\alpha}}\right] + 2
	\end{align*}
	where $\bar{\alpha} = \min\{1, \epsilon/\rho D^2\}$.
\end{corollary}
\noindent This matches Kiwiel's general convergence rate of $O(\|x_0-x^*\|^4/\epsilon^3)$ up to logarithmic terms.